\date{}
\documentclass[a4paper]{article}

\usepackage{setspace}
\usepackage{amsmath}
\usepackage{amsfonts}
\usepackage{cite}
\usepackage{amssymb}
\usepackage{amsthm}

\newtheorem{theorem}{Theorem}

\newtheorem{corollary}[theorem]{Corollary}

\newtheorem{lemma}[theorem]{Lemma}
\newtheorem{notation}[theorem]{Notation}

\begin{document}

\title{Fractional Multidimensional System}
\author{Xiaogang Zhu and Junguo Lu
\thanks{Junguo Lu is with the School of Electronic Information
and Electrical Engineering, Shanghai Jiao Tong University, Shanghai,
200240 China}
}

\maketitle

\section{Abstract}
The multidimensional ($n$-D) systems described by Roesser model are presented in this paper. These $n$-D systems
consist of discrete systems and continuous fractional order systems with fractional order $\nu$, $0<\nu<1$.
The stability and Robust stability of such $n$-D systems are investigated.

$\mathbf{Keywords}$: $n$-D; fractional; stability; Robust
\section{Introduction}

The multidimensional ($n$-D) systems have been studied for almost four decades \cite{Roesser1975discrete,Agathoklis1988Lyapunov,Galkowski2002LMI,Bachelier2008Kalman,Bachelier2012Fractional}. It
has been applied in fields such as image process \cite{Roesser1975discrete}, $n$-D coding and decoding \cite{Shi2002new} and $n$-D filtering \cite{Basu2002Multidimensional}.
The $n$-D systems can represent dynamic
processes that information propagates in many independent directions. However, the information of one dimensional systems only propagates in one direction.

As for a multidimensional system which consists of fractional order differential equations, Galkowski et al. first presented such a system in 2005 \cite{Galkowski2005Fractional}. But until now, researches on fractional $n$-D systems are either discrete system \cite{Galkowski2005Fractional} or continuous system with different fractional order \cite{Galkowski2005Fractional,Galkowski2006Fractional}. To the best of our knowledge, fractional $n$-D systems which consist of discrete system and fractional order system are not studied.

This paper focus on a hybrid $n$-D system which consists of discrete system and continuous fractional order system.

\begin{notation}
For a matrix $X$, $X^*,X^T$ denote the transpose conjugate and transpose of matrix $X$, respectively. $Sym(X)$ denotes
$X+X^*$.
$I$ is the identity matrix with appropriate
dimensions. For a matrix $X$, $X>0\ (X\geq0)$ means positive definite (semi-definite) and $X<0\ (X\leq0)$ means
negative definite (semi-definite). The notation $\mathcal{H}_n$ stands for the set of Hermitian matrices of dimension $n$.
And $\mathcal{H}_n^+\subset\mathcal{H}_n$ stands for the subset of positive definite matrices while $\mathcal{H}_n^-\subset\mathcal{H}_n$
is the subset of negative definite matrices.
Let the following notations be defined
\[
\underset{i=1}{\overset{k}{\oplus}}M_i=\underset{i=1,...,k}{\mathrm{diag}}M_i
\]
Let $\mathbb{I}(k)$ be
\[
\mathbb{I}(k)\triangleq{1,...,k},\ k\in\mathbb{N}
\]
\end{notation}

\section{Preliminaries}
Based on Bochniak's model\cite{Bochniak2005LMI}, Bachelier\cite{Bachelier2008Kalman} presented a hybrid version of Roesser model\cite{Roesser1975discrete}, which combined
integer order continuous system and discrete system. Here, we apply the Roesser model to a continuous-discrete fractional order system.
\begin{align}
\begin{bmatrix}
D^{\nu}x^1(t_1,...,t_r,j_{r+1},...,j_k)\\
\vdots\\
D^{\nu}x^r(t_1,...,t_r,j_{r+1},...,j_k)\\
x^{r+1}(t_1,...,t_r,j_{r+1}+1,...,j_k)\\
\vdots\\
x^{k}(t_1,...,t_r,j_{r+1},...,j_k+1)
\end{bmatrix}&=\begin{bmatrix}
A_{c} & A_{cd}\\
A_{dc} & A_{d}
\end{bmatrix}\begin{bmatrix}
x^1(t_1,...,t_r,j_{r+1},...,j_k)\\
\vdots\\
x^r(t_1,...,t_r,j_{r+1},...,j_k)\\
x^{r+1}(t_1,...,t_r,j_{r+1},...,j_k)\\
\vdots\\
x^{k}(t_1,...,t_r,j_{r+1},...,j_k)
\end{bmatrix}\notag\\
&+\begin{bmatrix}
B_1\\
B_2
\end{bmatrix}u(t_1,...,t_r,j_{r+1},...,j_k)\notag\\
y(t_1,...,t_r,j_{r+1},...,j_k)&=\begin{bmatrix}
C_1 & C_2
\end{bmatrix}\begin{bmatrix}
x^1(t_1,...,t_r,j_{r+1},...,j_k)\\
\vdots\\
x^r(t_1,...,t_r,j_{r+1},...,j_k)\\
x^{r+1}(t_1,...,t_r,j_{r+1},...,j_k)\\
\vdots\\
x^{k}(t_1,...,t_r,j_{r+1},...,j_k)
\end{bmatrix}\notag\\
&+Du(t_1,...,t_r,j_{r+1},...,j_k)\label{sysFraConDis}
\end{align}
where $\sum\limits_{i=1}^{k}n_i=n$ and $0<\nu\leq1$.

The vectors $x^i(t_1,...,t_r,j_{r+1},...,j_k)$, $u(t_1,...,t_r,j_{r+1},...,j_k)$ and $y(t_1,...,t_r,j_{r+1},...,j_k)$
are the local state subvectors, the input vector and the output vector, respectively. The matrices
\begin{align*}
&A=\begin{bmatrix}
A_{c} & A_{cd}\\
A_{dc} & A_{d}
\end{bmatrix}\in\mathbb{R}^{n\times n},\ B=\begin{bmatrix}
B_1\\
B_2
\end{bmatrix}\in\mathbb{R}^{n\times p},\\
&C=\begin{bmatrix}
C_1 & C_2
\end{bmatrix}\in\mathbb{R}^{l\times n},\ D\in\mathbb{R}^{l\times p}
\end{align*}
are the state, control, observation and transfer matrices respectively.

By applying the Laplace transform and the $Z$-transform of system (\ref{sysFraConDis}), the following
can be obtained
\[
Y(\rho)=T(\rho)U(\rho)
\]
where
\begin{equation}
T(\rho)=C(H(\rho)-A)^{-1}B+D,\ H(\rho)\triangleq\underset{i=1}{\overset{k}{\oplus}}\rho_{i}I_{n_i}
\label{defTHrho}
\end{equation}

Then, $\Delta(\rho,A)$ is defined by
\begin{equation}
\Delta(\rho,A)\triangleq\det(H(\rho)-A)\label{defDelta}
\end{equation}

In this paper, we use the Caputo's fractional derivative, of which the Laplace
transform allows utilization of initial values. The Caputo's fractional
derivative is defined as \cite{Podlubny1999Fractional}

\[
_{a}D_{t}^{\alpha}f(t)=\frac{1}{\Gamma(\alpha-n)}\int_{a}^{t}\frac
{f^{(n)}(\tau)d\tau}{(t-\tau)^{\alpha+1-n}}%
\]

where $n$ is an integer satisfying $0\leq n-1<\alpha<n$; $\Gamma(\cdot)$ is
the Gamma function which is defined as

\[
\Gamma(z)=\int_{0}^{\infty}e^{-t}t^{z-1}dt
\]

The following lemmas are useful for presenting our results.

\begin{lemma}\cite{Skeltonunified}
\label{lemNuXNuT}
For given matrices $\mathcal{U\in}\mathbb{C}^{n\times m},\Phi=\Phi^{\ast}\in\mathbb{C}^{n\times n}$, the following two statements are equivalent
\begin{enumerate}
\item $\mathcal{U}\Phi\mathcal{U}^*<0$
\item There exists a matrix $X$, such that $\Phi+\mathcal{N}_{u}X\mathcal{N}_{u}^{*}<0$
\end{enumerate}
where $\mathcal{N}_{u}$ is the orthogonal complement of $\mathcal{U}$.
\end{lemma}

\begin{lemma}\cite{Skeltonunified}
\label{lemOrthogonalComplement}
For given matrices $\mathcal{U}\in\mathbb{C}^{n\times m},\mathcal{V}\in\mathbb{C}^{k\times n},\Phi=\Phi^{\ast}\in\mathbb{C}^{n\times n}$, then the following two statements are equivalent
\begin{enumerate}
\item There exists matrix $\mathcal{X}\in\mathbb{C}^{m\times k}$ such that $Sym\left\{  \mathcal{U}X\mathcal{V}\right\}+\Phi<0$ holds
\item $\mathcal{N}_{u}\Phi\mathcal{N}_{u}^{*}<0$ and $\mathcal{N}_{v}^{*}\Phi\mathcal{N}_{v}<0$ hold
\end{enumerate}
where $\mathcal{N}_{u},\mathcal{N}_{v}^{*}$ are the orthogonal complement of $\mathcal{U},\mathcal{V}^{*}$, respectively.
\end{lemma}

\begin{lemma}\cite{Ortigueira2000Introduction}
\label{lemFinalVT}
Let $F(s)$ be the Laplace transform of the function $f(t)$, then for any $\nu>0$,
\[
\lim\limits_{t\rightarrow\infty}D^{\nu}f(t)=\lim\limits_{t\rightarrow0}s^{\nu+1}F(s),\ Re(s)>0
\]
\end{lemma}

\begin{lemma}\cite{Galkowski2002LMI}
\label{lemDiscStable}
A multidimensional discrete system
\begin{equation}
\begin{bmatrix}
x^{1}(j_{1}+1,...,j_k)\\
\vdots\\
x^{k}(j_{1},...,j_k+1)
\end{bmatrix}=A\begin{bmatrix}
x^{1}(j_{1},...,j_k)\\
\vdots\\
x^{k}(j_{1},...,j_k)
\end{bmatrix}
\end{equation}
is asymptotically stable if and only if
\begin{equation}
\det(H(z)-A)\neq0,\forall z\in \mathcal{U}_z
\end{equation}
where $H(z)$ is defined as in (\ref{defTHrho}) and $\mathcal{U}_z=\{z=\left[
z_1,...,z_k\right]^T:|z_i|\geq1,i=1,...,k\}$.
\end{lemma}

\section{Main results}
\subsection{Stability}

%\begin{lemma}\cite{Moze2005LMI}
%\label{lemContMoze}
%A continuous fractional system $D^{\nu}x(t)=Ax(t)$ with $0<\nu\leq1$ is asymptotically stable if and only if
%$|\arg(spec(A))|>\frac{\pi}{2}\nu$, where $spec(A)$ is the set of all eigenvalues of matrix $A$.
%\end{lemma}

\begin{lemma}\label{lemContStable}
A multidimensional continuous fractional order system with order $0<\nu\leq1$ and $\det(A)\neq0$
%$D^{\nu}x(t)=Ax(t)$
\begin{align}
\begin{bmatrix}
D^{\nu}x^1(t_1,...,t_r)\\
\vdots\\
D^{\nu}x^r(t_1,...,t_r)\\
\end{bmatrix}=A\begin{bmatrix}
x^1(t_1,...,t_r)\\
\vdots\\
x^r(t_1,...,t_r)
\end{bmatrix}\label{sysFracCont}
\end{align}
 is asymptotically stable if
\begin{equation}
\det(H(\lambda)-A)\neq0,\forall \lambda \in \mathcal{U}_\lambda
\end{equation}
where $H(\lambda)$ is defined as in (\ref{defTHrho}) and
$\mathcal{U}_\lambda=\{\lambda:\lambda=\left[\lambda_1,...,\lambda_r\right]^T,\ |\arg(\lambda_i)| \leq \frac{\pi}{2}\nu,0<\nu\leq1,i=1,...,r\}$.
\end{lemma}

\begin{proof}
%The system can be written into the form as
%\[
%D^{\nu}x(t)=Ax(t)
%\]
%where $t=\left[t_1,...,t_r\right]^T,x(t)=\left[x^1(t),...,x^r(t)\right]^T$.
%
%Sufficiency.If $\det(H(\lambda)-A)\neq0,\forall \lambda\in \mathcal{U}_\lambda$ holds.

Applying Laplace transform to the multidimensional fractional system
(\ref{sysFracCont}), the following holds
\begin{equation}
(H(s^\nu)-A)\begin{bmatrix}
X_1(s)\\
X_2(s)\\
\vdots\\
X_r(s)
\end{bmatrix}=\begin{bmatrix}
s^{\nu-1}x_1(0)\\
s^{\nu-1}x_2(0)\\
\vdots\\
s^{\nu-1}x_r(0)
\end{bmatrix}
\label{prfContSt1}
\end{equation}

It leads to
\begin{equation}
(H(s^\nu)-A)\begin{bmatrix}
s^{\nu+1}X_1(s)\\
s^{\nu+1}X_2(s)\\
\vdots\\
s^{\nu+1}X_r(s)
\end{bmatrix}=\begin{bmatrix}
s^{2\nu}x_1(0)\\
s^{2\nu}x_2(0)\\
\vdots\\
s^{2\nu}x_r(0)
\end{bmatrix}
\label{prfContSt2}
\end{equation}

$\forall\lambda\in \mathcal{U}_\lambda$, $\det(H(\lambda)-A)\neq0$, thus $\det(H(s^\nu)-A)\neq0$ when $Re(s)>0$. It means that (\ref{prfContSt2}) has the only solution
\begin{equation}
\begin{bmatrix}
s^{\nu+1}X_1(s)\\
s^{\nu+1}X_2(s)\\
\vdots\\
s^{\nu+1}X_r(s)
\end{bmatrix}=(H(s^\nu)-A)^{-1}\begin{bmatrix}
s^{2\nu}x_1(0)\\
s^{2\nu}x_2(0)\\
\vdots\\
s^{2\nu}x_r(0)
\end{bmatrix}
\end{equation}

Therefore, the following holds
\[
\lim\limits_{s\rightarrow0,Re(s)>0}s^{\nu+1}X_i(s)=0,\ i=1,2,...,r
\]

According to Lemma \ref{lemFinalVT},
\[
\lim\limits_{t\rightarrow+\infty}D^{\nu}x^{i}(t)=\lim\limits_{s\rightarrow0,Re(s)>0} s^{\nu+1}X_i(s)=0,\ i=1,...,r
\]

Therefore,
\[
\lim\limits_{t\rightarrow+\infty}x(t)=A^{-1}\lim\limits_{t\rightarrow+\infty}D^{\nu}x(t)=0
\]

It implies that the system is asymptotically stable.

%Necessity.
%
%If the system is asymptotically stable, which means
%\[
%\lim\limits_{t\rightarrow +\infty}x(t)=
%\lim\limits_{t\rightarrow +\infty}\left[x^1(t),...,x^r(t)\right]^T=0
%\]
%i.e. $\lim\limits_{t\rightarrow +\infty}x^{i}(t)=0,i=1,...,r$.
%
%According to Lemma \ref{lemContMoze}, if $\lim\limits_{t\rightarrow +\infty}x(t)=0$ holds, the inequality $|\arg(spec(A))|>\frac{\pi}{2}\nu$ holds.
%This inequality implies that when $\lambda \in \mathcal{U}_\lambda$, $\det(H(\lambda)-A)\neq0$.

This completes the proof.
\end{proof}

\begin{theorem}\label{thmStableRegion}
Consider a multidimensional system represented by (\ref{sysFraConDis}). Then, it's asymptotically stable if
\begin{equation}
\Delta(\rho,A)\neq0,\ \forall\rho\in \mathcal{U}_{z\lambda}
\end{equation}
where $\Delta(\cdot)$ is defined by (\ref{defDelta}) and $\mathcal{U}_{z\lambda}$ is defined by
\begin{align}
\mathcal{U}_{z\lambda}=\Bigg\{
\rho=\begin{bmatrix}
\rho_1\\
\vdots\\
\rho_k
\end{bmatrix}
\in\mathbb{C}^k : |\arg(\rho_i)|\leq\frac{\pi}{2}\nu,0<\nu\leq1,i=1,...,r; \notag\\
  |\rho_i|\geq1,i=r+1,...,k\Bigg\}
\end{align}
\end{theorem}

\begin{proof}
It can be proved by straightforward combinations of Lemma \ref{lemDiscStable} and \ref{lemContStable}.
\end{proof}

\subsection{Point-clustering}

%As Bachelier\cite{Bachelier2008Kalman} claimed, only the boundaries of $\mathcal{U}_{z\lambda}$ can be considered. The
%boundaries of $\mathcal{U}_{z\lambda}$ are two rays and a unit circle. The following will present the boundaries by
%equalities and inequalities.

To proceed, consider the following matrices
\begin{align}
P_i=\begin{bmatrix}
p_{i_{11}} & p_{i_{12}}\\
p_{i_{12}}^* & p_{i_{22}}
\end{bmatrix}\in\mathbb{C}^{2\times 2},\
Q_i=\begin{bmatrix}
q_{i_{11}} & q_{i_{12}}\\
q_{i_{12}}^* & q_{i_{22}}
\end{bmatrix}\in\mathbb{C}^{2\times 2}\label{defPiQi}
\end{align}

Define the sets $\mathcal{D}_i$ as
\begin{align}
\mathcal{D}_i\triangleq & \left\{
s\in\mathbb{C}: \mathcal{F}_{P_{i}}(s)\geq0,\mathcal{F}_{Q_{i}}(s)\geq0,\forall i\in\mathbb{I}(k)
\right\}\label{defDi}
\end{align}
where the functions $\mathcal{F}_{X_i}(s)$ are defined by
\begin{equation}
\mathcal{F}_{X_i}(s)\triangleq\begin{bmatrix}
sI\\
I
\end{bmatrix}^*X_i\begin{bmatrix}
sI\\
I
\end{bmatrix},\ \forall i\in\mathbb{I}(k)\label{defFXi}
\end{equation}

%Because the only meaningful sets are rays or circles,
We limit our consideration to sets described by $\mathcal{D}_i$.
Define the "$k$-region" $\mathcal{D}$ as
\begin{equation}
\mathcal{D}\triangleq\mathcal{D}_1\times\mathcal{D}_2\times ... \times\mathcal{D}_k
\label{defkD}
\end{equation}

Let $\mathcal{D}$ represent $\mathcal{U}_{z\lambda}$, then the matrices $P_i$ and $Q_i$ are
\begin{align}
P_{i} & =\begin{bmatrix}
0 & \sin(\frac{\pi}{2}\nu)-j\cos(\frac{\pi}{2}\nu)\\
\sin(\frac{\pi}{2}\nu)+j\cos(\frac{\pi}{2}\nu) & 0
\end{bmatrix}\notag \\
Q_{i} & =\begin{bmatrix}
0 & \sin(\frac{\pi}{2}\nu)+j\cos(\frac{\pi}{2}\nu)\\
\sin(\frac{\pi}{2}\nu)-j\cos(\frac{\pi}{2}\nu) & 0
\end{bmatrix}\label{equContPQ}
\end{align}
where $0<\nu\leq1$ and $\forall i\in\mathbb{I}(r)$.

And
\begin{align}
P_{i} & =\begin{bmatrix}
1 & 0\\
0 & -1
\end{bmatrix}\notag\\
Q_{i} & =\mathbf{0}_2\label{equDistPQ}
\end{align}
where $\forall i\in\{r+1,...,k\}$.

The following gives a sufficient condition for the stability of system (\ref{sysFraConDis}).

\begin{theorem}\label{thmCDStableLMI}
The system (\ref{sysFraConDis}) is asymptotically stable if there exist matrices $U_{n_i},V_{n_i}\in\mathcal{H}^+_{n_i}$ and a matrix $J=J^*$ such that
\begin{equation}
Z=G+\begin{bmatrix}
I & -A
\end{bmatrix}^*J\begin{bmatrix}
I & -A
\end{bmatrix}<0\label{ineqthmStable}
\end{equation}
where
\begin{equation}
G\triangleq\begin{bmatrix}
\underset{i=1}{\overset{k}{\oplus}}(U_{i}p_{i_{11}}+V_{i}q_{i_{11}}) & \underset{i=1}{\overset{k}{\oplus}}(U_{i}p_{i_{12}}+V_{i}q_{i_{12}})\\
\underset{i=1}{\overset{k}{\oplus}}(U_{i}p_{i_{12}}^*+V_{i}q_{i_{12}}^*) & \underset{i=1}{\overset{k}{\oplus}}(U_{i}p_{i_{22}}+V_{i}q_{i_{22}})
\end{bmatrix}\label{defG}
\end{equation}
and $p_{i_{11}},p_{i_{12}},p_{i_{22}},q_{i_{11}},q_{i_{12}},q_{i_{22}}$
are defined in (\ref{defPiQi}) with (\ref{equContPQ}) and (\ref{equDistPQ}).
\end{theorem}

\begin{proof}
We'll prove that for $\forall \rho$ that satisfies $\det(\underset{i=1}{\overset{k}{\oplus}}\rho_i I_{n_i}-A)=0$,
then $\rho\notin\mathcal{D}$.

Let $A(\rho)=\underset{i=1}{\overset{k}{\oplus}}\rho_i I_{n_i}-A$. If $\det(A(\rho))=0$,
then there exists a nonzero vector $y\in\mathbb{C}^n$ such that
\begin{equation}
A(\rho)y=0\label{prfStaLMI2}
\end{equation}

Let
\[
y=\begin{bmatrix}
y_1\\
\vdots\\
y_k
\end{bmatrix}
\]
where $y_i\in\mathbb{C}^{n_i}$.

And let
\[
x=\begin{bmatrix}
(\underset{i=1}{\overset{k}{\oplus}}\rho_i I_{n_i})y\\
y
\end{bmatrix}
\]

From (\ref{ineqthmStable}), we get
\[
x^*Zx<0
\]
which leads to
\begin{equation}
\sum\limits_{i=1}^{k}(\mathcal{F}_{P_i}(\rho_i)y^*_iU_iy_i+\mathcal{F}_{Q_i}(\rho_i)y^*_iV_iy_i)
+(A(\rho)y)^*J(A(\rho)y)<0\label{prfStaLMI1}
\end{equation}

According to (\ref{prfStaLMI2}), the second term of (\ref{prfStaLMI1}) is zero. If $\rho\in\mathcal{D}$, then according to (\ref{defDi}) the
first term of (\ref{prfStaLMI1}) is positive or zero. Therefore, when $\rho\in\mathcal{D}$, $\det(A(\rho))\neq0$.
Due to Theorem \ref{thmStableRegion}, the system (\ref{sysFraConDis}) is stable.
\end{proof}

\begin{corollary}\label{coroStable}
The system (\ref{sysFraConDis}) is asymptotically stable if there exist matrices $U_{n_i},V_{n_i}\in\mathcal{H}^+_{n_i}$ and a matrix $J=J^*$ such that
\begin{equation}
G+\begin{bmatrix}
I \\
-A
\end{bmatrix}J\begin{bmatrix}
I\\
-A
\end{bmatrix}^*<0\label{ineqcoroStable}
\end{equation}
where $G$ is defined as in (\ref{defG}).
\end{corollary}

\begin{proof}
Similar to the proof of Therem \ref{thmCDStableLMI}, if inequality (\ref{ineqcoroStable}) holds, then
\[
\det(\underset{i=1}{\overset{k}{\oplus}}\rho_i I_{n_i}-A^T)\neq0,\rho\in\mathcal{D}
\]
which is equivalent to
\[
\det(\underset{i=1}{\overset{k}{\oplus}}\rho_i I_{n_i}-A)\neq0,\rho\in\mathcal{D}
\]
Therefore, the system is asymptotically stable due to Theorem \ref{thmStableRegion}.
\end{proof}

\begin{corollary}\label{coroStableSym}
The system (\ref{sysFraConDis}) is asymptotically stable if there exist matrices $U_{n_i},V_{n_i}\in\mathcal{H}^+_{n_i}$ and a matrix $R$ such that
\begin{equation}
G+Sym\left\{\begin{bmatrix}
I \\ -A
\end{bmatrix}R\begin{bmatrix}
I & I
\end{bmatrix}\right\}<0\label{ineqcoroStableSym}
\end{equation}
where $G$ is defined as in (\ref{defG}).
\end{corollary}

\begin{proof}

Let
\begin{align}
G_{11} & \triangleq \underset{i=1}{\overset{k}{\oplus}}(U_{i}p_{i_{11}}+V_{i}q_{i_{11}})\notag\\
G_{12} & \triangleq \underset{i=1}{\overset{k}{\oplus}}(U_{i}p_{i_{12}}+V_{i}q_{i_{12}})\notag\\
G_{22} & \triangleq \underset{i=1}{\overset{k}{\oplus}}(U_{i}p_{i_{22}}+V_{i}q_{i_{22}})
\label{defG1122}
\end{align}

Then,
\begin{align}
&\begin{bmatrix}
I & -I
\end{bmatrix}G\begin{bmatrix}
I & -I
\end{bmatrix}^*\notag\\
=& G_{11}-G_{12}^*-G_{12}+G_{22}\notag\\
=& \begin{bmatrix}
-2\left( \underset{i=1}{\overset{r}{\oplus}}\left(U_{i}\sin(\frac{\pi}{2}\nu)+
V_{i}\sin(\frac{\pi}{2}\nu)\right)  \right) & 0\\
0 & \underset{i=r+1}{\overset{k}{\oplus}}0_{n_i}
\end{bmatrix}\notag\\
<& 0
\label{prfCoroStab2}
\end{align}

It's obvious that
\[
\begin{bmatrix}
I & -I
\end{bmatrix}\begin{bmatrix}
I\\ I
\end{bmatrix}=0
\]

Therefore, according to Lemma \ref{lemOrthogonalComplement} and inequality (\ref{ineqcoroStableSym}) and (\ref{prfCoroStab2}),
the following inequality holds
\begin{equation}
\begin{bmatrix}
A & I
\end{bmatrix}G\begin{bmatrix}
A^T \\
I
\end{bmatrix}<0\label{prfCoroStab1}
\end{equation}

According to Lemma \ref{lemNuXNuT}, inequality (\ref{prfCoroStab1}) is then equivalent to
\begin{equation}
G+\begin{bmatrix}
I \\ -A
\end{bmatrix}J\begin{bmatrix}
I & -A^T
\end{bmatrix}<0
\end{equation}

Thus, according to Corollary \ref{coroStable}, the system is asymptotically stable.

This completes the proof.

\end{proof}

\section{Numerical Examples}
\subsection{Example 1}
The following example presents a (1+1)D system of (\ref{sysFraConDis}), i.e. a system with one continuous independent variable and one discrete independent
variable. The system is considered:
\begin{align}
\nu=& 0.5\notag\\
A_c=& \begin{bmatrix}
-0.8 & 0\\
0 & -1.2
\end{bmatrix},\ A_{cd}=\begin{bmatrix}
0.5 & 0.3\\
0.7 & 0.2
\end{bmatrix}\notag\\
A_{dc}=& \begin{bmatrix}
0.4 & 0.3\\
0.8 & 0.9
\end{bmatrix},\ A_d=\begin{bmatrix}
-0.3 & 0\\
0 & -0.6
\end{bmatrix}\notag\\
A= & \begin{bmatrix}
A_c & A_{cd}\\
A_{dc} & A_{d}
\end{bmatrix}
%\notag\\
%B=&\begin{bmatrix}
%0.003&0.009&0.008&0.004\\
%0.012&0.005&0.009&0.014\\
%0.007&0.006&0.011&0.009\\
%0.011&0.013&0.011&0.013
%\end{bmatrix}\notag\\
%C=&\begin{bmatrix}
%15&12&19&2\\
%9&1&3&7\\
%18&11&8&12\\
%11&7&2&6
%\end{bmatrix}\notag\\
%D=&\begin{bmatrix}
%0.005&0.008&0.003&0.001\\
%0.007&0.005&0.009&0.007\\
%0.007&0.002&0.003&0.001\\
%0.006&0.005&0.009&0.005
%\end{bmatrix}
\label{paraSysCD}
\end{align}

Let a system be system (\ref{sysFraConDis}) with (\ref{paraSysCD}). Applying Theorem \ref{thmCDStableLMI}, the variables can be calculated by the Matlab LMI toolbox. The solution is
\begin{align}
U_1=&\begin{bmatrix}
146.84 & 0 \\
0 & 146.84
\end{bmatrix},\ U_2=\begin{bmatrix}
24.3 & 5.99\\
5.99 & 1.9
\end{bmatrix}\notag\\
V_1=&\begin{bmatrix}
4.24 & 14.68\\
14.68 & 194.82
\end{bmatrix},\ V_2=\begin{bmatrix}
146.84 & 0 \\
0 & 146.84
\end{bmatrix}\notag\\
%G=&\begin{bmatrix}
% 0 & 0 & 0 & 0 & 106.8 - 100.8i & 10.3 + 10.3i & 0 & 0 \\
% 0 & 0 & 0 & 0 & 10.3 + 10.3i & 241.5 + 33.9i &0 & 0\\
% 0 & 0 & 24.3 & 5.9 & 0 & 0 & 0 & 0 \\
% 0 & 0 & 5.9 &	1.9 & 0 & 0 & 0 & 0\\
% 106.8+100.8i & 10.3 - 10.3i & 0 & 0 & 0 & 0 &	0 & 0\\
% 10.3 - 10.3i & 241.5 - 33.9i & 0 & 0 & 0 & 0 &	0 & 0\\
% 0 & 0 & 0 & 0 & 0 & 0 & -24.3 &	-5.9 \\
% 0 & 0 & 0 & 0 & 0 & 0 & -5.9 & -1.9
%\end{bmatrix}\notag\\
J=&\begin{bmatrix}
-164.9 & -57.19 & -210.42 & -75.7\\
-57.19 & -108.77 & -154.38 & -62.71\\
-210.42 & -154.38 & -643.18 & -137.73\\
-75.7 & -62.71 & -137.73 & -116.65
\end{bmatrix}
\end{align}

It means that the continuous-discrete (1+1)D system is stable.

\subsection{Example 2}

The system is considered:
\begin{align}
\nu=& 0.9\notag\\
A_c=& \begin{bmatrix}
-0.8 & 0.5\\
0.3 & -1.2
\end{bmatrix},\ A_{cd}=\begin{bmatrix}
0.5 & 0.6\\
0.7 & 0.8
\end{bmatrix}\notag\\
A_{dc}=& \begin{bmatrix}
0.9 & 0.1\\
0.2 & 0.1
\end{bmatrix},\ A_d=\begin{bmatrix}
-0.7 & 0\\
0 & -0.2
\end{bmatrix}\notag\\
A= & \begin{bmatrix}
A_c & A_{cd}\\
A_{dc} & A_{d}
\end{bmatrix}
\label{paraSysCD2}
\end{align}

Let a system be system (\ref{sysFraConDis}) with (\ref{paraSysCD2}). Applying Corollary \ref{coroStableSym}, the variables can be calculated by the Matlab LMI toolbox. The solution is
\begin{align}
U_1=&\begin{bmatrix}
35333.44 & 0 \\
0 & 35333.44
\end{bmatrix},\ U_2=\begin{bmatrix}
36674.54 & 4958.14\\
4958.14 & 7924.014
\end{bmatrix}\notag\\
V_1=&\begin{bmatrix}
32747.79 & 0\\
0 & 14331.51
\end{bmatrix},\ V_2=\begin{bmatrix}
51948.27 & 29479.11\\
29479.11 & 51948.27
\end{bmatrix}\notag\\
R=&\begin{bmatrix}
-9111.82 & -53.82 &	-29013.36 &	-830.36\\
1813.13 & -12214.89 &	-13744.63 &	-6499.66 \\
26731.07 &	4233.10 &	-30626.82 &	-11869.47\\
2051.31 &	1540.31 &	2773.74 &	-8884.96\\
\end{bmatrix}
\end{align}

It means that the continuous-discrete (1+1)D system is stable.

\section{Conclusion}

In this paper, the fractional continuous-discrete systems are presented, where the fractional order is $0<\nu\leq1$. The stability and Robust stability of fractional continuous-discrete systems have been investigated.
Invoking fractional final value theorem, the sufficient condition of stability of such systems is proved. Then, we
prove the sufficient condition of Robust multidimensional interval system. Finally, examples are given to verify the theorems.

\bibliographystyle{unsrt}
\bibliography{Multidimensional_system}
\end{document}